\documentclass[12pt]{amsart}
\newtheorem{theorem}{Theorem}[section]
\newtheorem{lemma}[theorem]{Lemma}
\newtheorem{proposition}[theorem]{Proposition}

 \theoremstyle{definition}

\theoremstyle{remark}
\newtheorem{remark}[theorem]{Remark}

\numberwithin{equation}{section}

\begin{document}
\setlength{\baselineskip}{1.2\baselineskip}

\title[ mixed Hessian equations]
{The Classical Neumann Problem for a class of mixed Hessian equations}
\author{Chuanqiang Chen}
\address{School of Mathematics and Statistics, Ningbo University, Ningbo,  315211, P.R. China}
\email{chenchuanqiang@nbu.edu.cn}

\author{Li Chen}
\address{Faculty of Mathematics and Statistics, Hubei Key Laboratory of Applied Mathematics, Hubei University,  Wuhan 430062, P.R. China}
\email{chernli@163.com}

\author{Ni Xiang}
\address{Faculty of Mathematics and Statistics, Hubei Key Laboratory of Applied Mathematics, Hubei University,  Wuhan 430062, P.R. China}
\email{nixiang@hubu.edu.cn}
\thanks{Research of the first author was supported by NSFC No.11771396, and research of the second and the third authors was supported by funds from Hubei Provincial Department of Education Key Projects D20171004, D20181003 and the National Natural Science Foundation of China No.11971157.}

\begin{abstract}
In this paper, we establish global $C^2$ estimates for a class of mixed Hessian equations with Neumann boundary condition, and obtain the existence theorem of $k$-admissible solutions for the classical Neumann problem of these mixed Hessian equations.



\end{abstract}

\maketitle
\bigskip

\section{Introduction}

\medskip
In this paper, we consider the classical Neumann problem for the following mixed Hessian equations
\begin{equation}\label{1.1}
\sigma_k(D^2u) = \sum _{l=0}^{k-1} \alpha_l(x) \sigma_l (D^2u),  \text{ in~} \Omega,
\end{equation}
where $k\geq 2$, $\Omega \subset \mathbb{R}^n$ is a bounded domain, $D^2 u$ is the Hessian matrix of the function $u$, $\alpha_l(x)>0$ in $\overline{\Omega}$ with $l=0, 1, \cdots, k-1$, are given positive functions in $\overline{\Omega}$, and for any $m = 1, \cdots, n$,
\begin{equation*}
\sigma_m(D^2 u) = \sigma_m(\lambda(D^2 u)) = \sum _{1 \le i_1 < i_2 <\cdots<i_m\leq n}\lambda_{i_1}\lambda_{i_2}\cdots\lambda_{i_m},
\end{equation*}
with $\lambda(D^2 u) =(\lambda_1,\cdots,\lambda_n)$ being the eigenvalues of $D^2 u$. We also set $\sigma_0=1$.
Recall that the G{\aa}rding's cone is defined as
\begin{equation*}
\Gamma_k  = \{ \lambda  \in \mathbb{R}^n :\sigma _i (\lambda ) > 0,\forall 1 \le i \le k\}.
\end{equation*}
If $\lambda(D^2 u) \in \Gamma_k$ for any $x \in \Omega$, we say $u$ is a $k$-admissible function.

The equation \eqref{1.1} is a general class of mixed Hessian equation. Specially, it is Monge-Amp\`{e}re equation when $k=n$, $\alpha_0(x) >0$ and $\alpha_1(x) = \cdots = \alpha_{n-1}(x) \equiv 0$, $k$-Hessian equation when $\alpha_0(x) >0$ and $\alpha_1(x) = \cdots = \alpha_{k-1}(x) \equiv 0$, and Hessian quotient equation when $\alpha_{m}(x) >0$ ($k-1\geq m>0$) and $\alpha_0(x)= \cdots= \alpha_{m-1}(x)= \alpha_{m+1}(x) = \cdots= \alpha_{k-1}(x) \equiv 0$. This kind of equations is motivated from the study of many important geometric problems. For example, the problem of prescribing convex combination of area measures was proposed in \cite{S}, which leads to mixed Hessian equations of the form
\[
\sigma _k (\nabla^2 u + uI_n ) + \sum\limits_{i = 0}^{k - 1} {\alpha _i \sigma _i (\nabla^2 u + uI_n )}  = \phi (x), x \in  \mathbb{S}^n.
\]
The special Lagrangian equation introduced by Harvey-Lawson \cite{HL82} in the study of calibrated geometries is also a mixed type Hessian equation
\[
{\mathop{\rm Im}\nolimits} \det (I_{2n}  + \sqrt { - 1} D^2 u) = \sum\limits_{k = 0}^{[(n - 1)/2]} {( - 1)^k \sigma _{2k + 1} (D^2 u)}  = 0.
\]
Another important example is Fu-Yau equation in \cite{FY2007,FY2008} arising from the study of the Hull-Strominger system in theoretical physics, which is an equation that can be written as the linear combination of the first and the second elementary symmetric functions
\[
\sigma _1 (i\partial \overline \partial  (e^u  + \alpha 'e^{ - u} )) + \alpha '\sigma _2 (i\partial \overline \partial  u) = 0.
\]

For the Dirichlet problem of elliptic equations in $\mathbb{R}^n$, many results are well known. For example, the Dirichlet problem of the Laplace equation was studied in \cite{GT}. Caffarelli-Nirenberg-Spruck \cite{CNS84} and Ivochkina \cite{I87} solved the Dirichlet problem of the Monge-Amp\`{e}re equation. Caffarelli-Nirenberg-Spruck \cite{CNS85} solved the Dirichlet problem of the $k$-Hessian equation. For the general Hessian quotient equation, the Dirichlet problem was solved by Trudinger in \cite{T95}.

Also, the Neumann or oblique derivative problem of partial differential equations  has been widely studied. For a priori estimates and the existence theorem of Laplace equation with Neumann boundary condition, we refer to the book \cite{GT}. Also, we can see the recent book written by Lieberman
\cite{L13} for the Neumann or oblique derivative problem of linear and quasilinear elliptic equations. In 1986, Lions-Trudinger-Urbas solved the Neumann problem of the Monge-Amp\`{e}re equation in the celebrated paper \cite{LTU86}. For related results on the Neumann or oblique derivative problem for some class of fully nonlinear elliptic equations can be found in Urbas \cite{U95} and \cite{U96}. For the Neumann problem of $k$-Hessian equations,  Trudinger \cite{T87} established the existence theorem when the domain is a ball, and Ma-Qiu \cite{MQ15} and Qiu-Xia \cite{QX16} solved the strictly convex domain case. D.K. Zhang and the first author \cite{CZ16} solved the Neumann problem of general Hessian quotient equations. Jiang and Trudinger \cite{JT15,JT16,JT3}, studied the general oblique boundary problem for augmented Hessian equations with some regular conditions and concavity conditions.

Krylov in \cite{Kr} considered the Dirichlet problem of \eqref{1.1} with $\alpha_l(x)\geq 0$ for $0\leq l\leq k-1$, and he observed that the natural admissible cone to make equation elliptic is also the $\Gamma_k$. Recently, Guan-Zhang in \cite{GZ2019} considered the $(k-1)$-admissible solution without the sign of $\alpha_{k-1}$ and obtained the global $C^2$ estimates.

Naturally, we want to know how about the classical Neumann problem of \eqref{1.1}. In this paper, we obtain the existence theorem as follows,
\begin{theorem} \label{th1.1}
Suppose that $\Omega \subset \mathbb{R}^n$ is a $C^4$ strictly convex domain, $2 \leq k \leq n$, $\nu$ is the outer unit normal vector of $\partial \Omega$, $\alpha_l(x) \in C^2(\overline{\Omega})$ with $l=0, 1, \cdots, k-1$ are positive functions and $\varphi \in C^3(\partial \Omega)$. Then there exists a unique constant $c$, such that the classical Neumann problem
\begin{align} \label{1.2}
\left\{ \begin{array}{l}
\sigma _k (D^2 u) = \sum\limits_{l=0}^{k-1} \alpha_l(x) \sigma_l (D^2u),  \quad \text{in} \quad \Omega,\\
u_\nu = c + \varphi(x),\qquad \text{on} \quad \partial \Omega,
 \end{array} \right.
\end{align}
has $k$-admissible solutions $u \in C^{3, \alpha}(\overline \Omega)$, which are unique up to a constant.
\end{theorem}

\begin{remark} \label{rmk1.2}
For the classical Neumann problem of mixed Hessian equations \eqref{1.2}, it is easy to know that a solution plus any constant is still a solution. So we cannot obtain a uniform bound for the solutions of \eqref{1.2}, and cannot use the method of continuity directly to get the existence. As in Lions-Trudinger-Urbas \cite{LTU86} (see also Qiu-Xia \cite{QX16}), we consider the $k$-admissible solution $u^\varepsilon$ of the approximation equation
\begin{align}\label{1.3}
\left\{ \begin{array}{l}
\sigma _k (D^2 u) = \sum\limits_{l=0}^{k-1} \alpha_l(x) \sigma_l (D^2u),  \quad \text{in} \quad \Omega,\\
u_\nu = - \varepsilon u  + \varphi(x),\qquad \text{on} \quad \partial \Omega,
 \end{array} \right.
\end{align}
for any small $\varepsilon >0$. We need to establish a priori estimates of $u^\varepsilon$ independent of $\varepsilon$, and then we can obtain a solution of \eqref{1.2} by letting $\varepsilon \rightarrow 0$ and a perturbation argument. The uniqueness holds from the maximum principle and Hopf Lemma.
\end{remark}

The rest of this paper is organized as follows. In Section 2, we give some definitions and important lemmas. In Section 3, we prove global $C^2$ estimates of \eqref{1.3}. In Section 4, we give the proof for the existence, that is Theorem \ref{th1.1}.

\section{Preliminaries}

In this section, we give some basic properties of elementary symmetric functions, which could be found in
\cite{L96}, and establish some key lemmas.

\subsection{Basic properties of elementary symmetric functions}

First, we denote by $\sigma _m (\lambda \left| i \right.)$ the symmetric
function with $\lambda_i = 0$ and $\sigma _m (\lambda \left| ij
\right.)$ the symmetric function with $\lambda_i =\lambda_j = 0$.
\begin{proposition}\label{prop2.1}
Let $\lambda=(\lambda_1,\dots,\lambda_n)\in\mathbb{R}^n$ and $m
= 1, \cdots,n$, then
\begin{align*}
&\sigma_m(\lambda)=\sigma_m(\lambda|i)+\lambda_i\sigma_{m-1}(\lambda|i), \quad \forall \,1\leq i\leq n,\\
&\sum_i \lambda_i\sigma_{m-1}(\lambda|i)=m\sigma_{m}(\lambda),\\
&\sum_i\sigma_{m}(\lambda|i)=(n-m)\sigma_{m}(\lambda).
\end{align*}
\end{proposition}

We also denote by $\sigma _m (W \left|
i \right.)$ the symmetric function with $W$ deleting the $i$-row and
$i$-column and $\sigma _m (W \left| ij \right.)$ the symmetric
function with $W$ deleting the $i,j$-rows and $i,j$-columns. Then
we have the following identities.
\begin{proposition}\label{prop2.2}
Suppose $W=(W_{ij})$ is diagonal, and $m$ is a positive integer,
then
\begin{align*}
\frac{{\partial \sigma _m (W)}} {{\partial W_{ij} }} = \begin{cases}
\sigma _{m - 1} (W\left| i \right.), &\text{if } i = j, \\
0, &\text{if } i \ne j.
\end{cases}
\end{align*}
\end{proposition}

Recall that the G{\aa}rding's cone is defined as
\begin{equation}\label{2.1}
\Gamma_m  = \{ \lambda  \in \mathbb{R}^n :\sigma _i (\lambda ) >
0,\forall 1 \le i \le m\}.
\end{equation}

\begin{proposition}\label{prop2.3}
Let $\lambda=(\lambda_1,\dots,\lambda_n) \in \Gamma_m$ and $m \in \{1,2, \cdots, n\}$. Suppose that
$$
\lambda_1 \geq \cdots \geq \lambda_m \geq \cdots \geq \lambda_n,
$$
then we have
\begin{align}
\label{2.2}& \sigma_{m-1} (\lambda|n) \geq \sigma_{m-1} (\lambda|n-1) \geq \cdots \geq \sigma_{m-1} (\lambda|m) \geq \cdots \geq \sigma_{m-1} (\lambda|1) >0; \\
\label{2.3}& \lambda_1 \geq \cdots \geq \lambda_m >  0, \quad \sigma _m (\lambda)\leq C_n^m  \lambda_1 \cdots \lambda_m; \\
\label{2.4}& \lambda _1 \sigma _{m - 1} (\lambda |1) \geq \frac{m} {{n}}\sigma _m (\lambda), \\
\label{2.5}& \sigma _{m - 1} (\lambda |m) \geq c(n,m)\sigma_{m-1} (\lambda),
\end{align}
where $C_n^m = \frac{n!}{m! (n-m)!}$.
\end{proposition}

\begin{proof}
All the properties are well known. For example, see \cite{L96} or \cite{HS99} for a proof of \eqref{2.2},
\cite{L91} for \eqref{2.3}, \cite{CW01} or \cite{HMW10} for \eqref{2.4}, and \cite{LT94} for \eqref{2.5}.
\end{proof}

The generalized Newton-MacLaurin inequality is as follows, which will be used all the time.
\begin{proposition}\label{prop2.4}
For $\lambda \in \Gamma_m$ and $m > l \geq 0$, $ r > s \geq 0$, $m \geq r$, $l \geq s$, we have
\begin{align}
\Bigg[\frac{{\sigma _m (\lambda )}/{C_n^m }}{{\sigma _l (\lambda )}/{C_n^l }}\Bigg]^{\frac{1}{m-l}}
\le \Bigg[\frac{{\sigma _r (\lambda )}/{C_n^r }}{{\sigma _s (\lambda )}/{C_n^s }}\Bigg]^{\frac{1}{r-s}}. \notag
\end{align}
\end{proposition}
\begin{proof}
See \cite{S05}.
\end{proof}

\subsection{Key Lemmas}

In the establishment of the a priori estimates, the following inequalities and properties play an important role.

For the convenience of notations, we will denote
\begin{equation}\label{2.6}
G_k(D^2u):= \frac{\sigma_k(D^2u)}{\sigma_{k-1}(D^2u)},\ \ G_l(D^2u) := -\frac{\sigma_l(D^2u)}{\sigma_{k-1}(D^2u)},~ 0\leq l\leq k-2,
\end{equation}
\begin{equation}\label{2.7}
G(D^2 u, x):= G_k(D^2 u) + \sum_{l=0}^{k-2} \alpha_l(x) G_l(D^2u),
\end{equation}
and
\begin{equation}\label{2.8}
G^{ij} :=\frac{\partial G}{\partial u_{ij}}, ~~ 1 \leq i, j \leq n.
\end{equation}

\begin{lemma}\label{lem2.5}
If $u$ is a $C^2$ function with $\lambda(D^2 u)\in \Gamma_{k}$, and $\alpha_l(x)$ ($0 \leq l \leq k-2$) are positive, then the operator $G$ is elliptic and concave.
\end{lemma}
\begin{proof}
The lemma holds for $\lambda(D^2 u)\in \Gamma_{k-1}$ (see the proof in \cite{GZ2019}).
\end{proof}

\begin{lemma}\label{lem2.6}
If $u$ is a $k$-admissible solution of \eqref{1.1}, and $\alpha_l(x)$ ($0 \leq l \leq k-1$) are positive, then
\begin{align}
\label{2.9}& 0 <\frac{\sigma_l(D^2u)}{\sigma_{k-1}(D^2u)} \leq C(n,k, \inf \alpha_l), ~~ 0 \leq l \leq k-2; \\
\label{2.10}& 0< \inf \alpha_{k-1} \leq \frac{\sigma_k(D^2u)}{\sigma_{k-1}(D^2u)} \leq C(n,k, \sum\limits_{l=0}^{k-1}\sup \alpha_l).
\end{align}
\end{lemma}
\begin{proof}
The left hand sides of \eqref{2.9} and \eqref{2.10} are easy to prove. In the following, we prove the right hand sides.

Firstly, if $\frac{\sigma_k}{\sigma_{k-1}}\leq 1$, then we get from the equation \eqref{1.1}
\begin{equation*}
\alpha_l \frac{\sigma_l}{\sigma_{k-1}} \leq \frac{\sigma_k}{\sigma_{k-1}}  \leq 1, ~~0 \leq l \leq k-2.
\end{equation*}

Secondly, if $\frac{\sigma_k}{\sigma_{k-1}} > 1$, i.e. $\frac{\sigma_{k-1}}{\sigma_{k}} < 1$.
We can get for $0 \leq l \leq k-2$ by the Newton-MacLaurin inequality,
\begin{equation*}
\frac{\sigma_l}{\sigma_{k-1}}\leq \frac{(C_n^k)^{k-1-l}C_n^l}{(C_n^{k-1})^{k-l}}(\frac{\sigma_{k-1}}{\sigma_k})^{k-1-l} \leq \frac{(C_n^k)^{k-1-l}C_n^l}{(C_n^{k-1})^{k-l}} \leq C(n,k),
\end{equation*}
and
\begin{equation*}
\frac{\sigma_k}{\sigma_{k-1}} = \sum\limits_{l=0}^{k-1} \alpha_l \frac{\sigma_l}{\sigma_{k-1}} \leq  C(n,k) \sum\limits_{l=0}^{k-1} \sup \alpha_l.
\end{equation*}
\end{proof}

\begin{lemma}\label{lem2.7}
If $u$ is a $k$-admissible solution of \eqref{1.1}, and $\alpha_l(x)$ ($0 \leq l \leq k-1$) are positive, then
\begin{align}
\label{2.11}& \frac{n-k+1}{k} \leq \sum G^{ii} < n-k-1; \\
\label{2.12}&  \inf \alpha_{k-1} \leq  \sum G^{ij} u_{ij} \leq C(n, k, \sum\limits_{l=0}^{k-1} \sup \alpha_l).
\end{align}
\end{lemma}

\begin{proof}
By direct computations, we can get
\begin{align}
\sum {G^{ii} }  \ge \sum {\frac{{\partial \left( {\frac{{\sigma _k }}{{\sigma _{k - 1} }}} \right)}}{{\partial \lambda _i }}}  =& \sum {\frac{{\sigma _{k - 1} (\lambda |i)\sigma _{k - 1}  - \sigma _k \sigma _{k - 2} (\lambda |i)}}{{\sigma _{k - 1} ^2 }}}  \notag \\
=& \frac{{(n - k + 1)\sigma _{k - 1} ^2  - (n - k + 2)\sigma _k \sigma _{k - 2} }}{{\sigma _{k - 1} ^2 }} \notag \\
\ge& \frac{{n - k + 1}}{k},
\end{align}
and
\begin{align}
\sum {G^{ii} }  =& \sum {\frac{{\partial \left( {\frac{{\sigma _k }}{{\sigma _{k - 1} }}} \right)}}{{\partial \lambda _i }}}  - \sum\limits_{l = 0}^{k - 2} {\alpha _l \sum\limits_i {\frac{{\partial \left( {\frac{{\sigma _l }}{{\sigma _{k - 1} }}} \right)}}{{\partial \lambda _i }}} }  \notag \\
=& \sum {\frac{{\sigma _{k - 1} (\lambda |i)\sigma _{k - 1}  - \sigma _k \sigma _{k - 2} (\lambda |i)}}{{\sigma _{k - 1} ^2 }}}  - \sum\limits_{l = 0}^{k - 2} {\alpha _l \sum\limits_i {\frac{{\sigma _{l - 1} (\lambda |i)\sigma _{k - 1}  - \sigma _l \sigma _{k - 2} (\lambda |i)}}{{\sigma _{k - 1} ^2 }}} } \notag \\
=& \frac{{(n - k + 1)\sigma _{k - 1} ^2  - (n - k + 2)\sigma _k \sigma _{k - 2} }}{{\sigma _{k - 1} ^2 }} \notag \\
 &+ \sum\limits_{l = 0}^{k - 2} {\alpha _l \frac{{(n - k + 2)\sigma _l \sigma _{k - 2}  - (n - l + 1)\sigma _{l - 1} \sigma _{k - 1} }}{{\sigma _{k - 1} ^2 }}}  \notag \\
\le& (n - k + 1) - \frac{{(n - k + 2)\sigma _{k - 2} }}{{\sigma _{k - 1} }}\left( {\frac{{\sigma _k }}{{\sigma _{k - 1} }} - \sum\limits_{l = 0}^{k - 2} {\alpha _l \frac{{\sigma _l }}{{\sigma _{k - 1} }}} } \right)  \notag \\
<& n - k + 1,
\end{align}
hence \eqref{2.11} holds. Also, we can get
\begin{align}
\sum {G^{ij} u_{ij} }  =& \sum {\frac{{\partial \left( {\frac{{\sigma _k }}{{\sigma _{k - 1} }}} \right)}}{{\partial \lambda _i }}\lambda _i }  - \sum\limits_{l = 0}^{k - 2} {\alpha _l \sum\limits_i {\frac{{\partial \left( {\frac{{\sigma _l }}{{\sigma _{k - 1} }}} \right)}}{{\partial \lambda _i }}\lambda _i } }  \notag \\
=& \frac{{\sigma _k }}{{\sigma _{k - 1} }} + \sum\limits_{l = 0}^{k - 2} {(k - 1 - l)\alpha _l \frac{{\sigma _l }}{{\sigma _{k - 1} }}}  \notag \\
=& \alpha _{k-1}+ \sum\limits_{l = 0}^{k - 2} {(k  - l)\alpha _l \frac{{\sigma _l }}{{\sigma _{k - 1} }}},
\end{align}
hence \eqref{2.12} holds.

\end{proof}

The following lemmas play an important role in the proof of a priori estimates. The idea of the proof for these lemmas comes from the paper in \cite{CZ16}.

\begin{lemma}\label{lem2.8}
Suppose $\lambda=(\lambda_1,\lambda_2,\cdots,\lambda_n)\in \Gamma_{k}$, $k\geq 2$, and $\lambda_1<0$. Then we have
\begin{equation}
\frac{\partial G}{\partial \lambda_1}\geq \frac{n}{k}\frac{1}{(n-k+2)^2}\sum_{i=1}^{n}\frac{\partial G}{\partial \lambda_i},
\end{equation}
where $G(\lambda) := \frac{{\sigma _k(\lambda) }}{{\sigma _{k - 1}(\lambda) }} - \sum\limits_{l = 0}^{k - 2} {\alpha _l \frac{{ \sigma _l (\lambda)}}{{\sigma _{k - 1}(\lambda)}}} $ with $\alpha_l $ positive.
\end{lemma}
\begin{proof}
The lemma holds from the Lemma 2.5 in \cite{CZ16} and the following fact
\begin{equation}
\frac{{\partial G}}{{\partial \lambda _i }} = \frac{{\partial \left( {\frac{{\sigma _k }}{{\sigma _{k - 1} }}} \right)}}{{\partial \lambda _i }}{\rm{ + }}\sum\limits_{l = 0}^{k - 2} {\alpha _l \frac{1}{{\left( {\frac{{\sigma _{k - 1} }}{{\sigma _l }}} \right)^2 }}\frac{{\partial \left( {\frac{{\sigma _{k - 1} }}{{\sigma _l }}} \right)}}{{\partial \lambda _i }}}.
\end{equation}

\end{proof}

\begin{lemma}\label{lem2.9}
Suppose $\lambda=(\lambda_1,\lambda_2,\cdots,\lambda_n)\in \Gamma_k$, $k\geq 2$, and $\lambda_2\geq \cdots \geq \lambda_n$. If $\lambda_1>0$, $\lambda_n<0$, $\lambda_1\geq \delta \lambda_2$, and $-\lambda_n\geq \epsilon\lambda_1$ for small positive constants $\delta$ and $\epsilon$, then we have
\begin{equation}\label{2.18}
\frac{\partial G}{\partial \lambda_1}\geq c_2 \sum_{i=1}^n \frac{\partial G}{\partial \lambda_1},
\end{equation}
where $c_2= \frac{n}{k}\frac{c_1^2}{(n-k+2)^2}$ with $c_1=\min\{\frac{\epsilon^2\delta^2}{2(n-2)(n-1)},\frac{\epsilon^2\delta}{4(n-1)}\}$.
\end{lemma}
\begin{proof}
The lemma holds from the Lemma 2.7 in \cite{CZ16} and the following fact
\begin{equation}
\frac{{\partial G}}{{\partial \lambda _i }} = \frac{{\partial \left( {\frac{{\sigma _k }}{{\sigma _{k - 1} }}} \right)}}{{\partial \lambda _i }}{\rm{ + }}\sum\limits_{l = 0}^{k - 2} {\alpha _l \frac{1}{{\left( {\frac{{\sigma _{k - 1} }}{{\sigma _l }}} \right)^2 }}\frac{{\partial \left( {\frac{{\sigma _{k - 1} }}{{\sigma _l }}} \right)}}{{\partial \lambda _i }}}.
\end{equation}

\end{proof}

\section{a priori estimates of the approximation equation \eqref{1.3}}

In this section, we prove the $C^2$ a priori estimates of $k$-admissible solutions of the approximation equation \eqref{1.3}, including the $C^0$ estimate, global gradient estimate and global second order derivatives estimate.

\subsection{$C^0$ estimate}

The $C^0$ estimate is easy. For completeness, we produce a proof here following the idea of Lions-Trudinger-Urbas \cite{LTU86}.

\begin{theorem} \label{th3.1}
Suppose $\Omega \subset \mathbb{R}^n$ is a $C^1$ bounded domain, $\alpha_l(x) \in C^0(\overline{\Omega})$ with $l=0, 1, \cdots, k-1$ are positive functions and $\varphi \in C^0(\partial \Omega)$, and $u \in C^2(\Omega)\cap C^1(\overline \Omega)$ is the $k$-admissible solution of the equation \eqref{1.3} with $\varepsilon \in (0,1)$, then we have
\begin{align}\label{3.1}
\sup_\Omega |\varepsilon u|  \leq M_0,
\end{align}
where $M_0$ depends on $n$, $k$, $ \text{diam} (\Omega)$, $\max\limits_{\partial \Omega} |\varphi|$ and $\sum\limits_{l=0}^{k-1} \sup\limits_\Omega \alpha_l$.
\end{theorem}

\begin{proof}
Firstly, since $u$ is subharmonic, the maximum of $u$ is attained at some boundary point $x_0 \in \partial \Omega$. Then we can get
\begin{align}\label{3.2}
0 \leq u_\nu (x_0) =-\varepsilon u(x_0) + \varphi (x_0).
\end{align}
Hence
\begin{align}\label{3.3}
\max_{\overline{\Omega}}( \varepsilon u )= \varepsilon u(x_0) \leq \varphi (x_0) \leq \max_{\partial \Omega} |\varphi|.
\end{align}

For a fixed point $x_1 \in \Omega$, and a positive constant $A$ large enough, we have
\begin{align} \label{3.4}
G (D^2 (A|x-x_1|^2), x) =& 2A \frac{C_n^k}{C_n^{k-1}} - \sum\limits_{l=0}^{k-2} \alpha_l (2A)^{-(k-1-l)} \frac{C_n^l}{C_n^{k-1}} \notag \\
\geq& \sup_\Omega \alpha_{k-1} \geq \alpha_{k-1}(x) = G (D^2 u, x).
\end{align}
By the comparison principle, we know $u - A|x-x_1|^2$ attains its minimum at some boundary point $x_2 \in \partial \Omega$. Then
\begin{align}\label{3.5}
0 \geq& (u - A|x-x_1|^2)_\nu (x_2) = u_\nu (x_2) -2A (x_2-x_1)\cdot \nu  \notag \\
=&-\varepsilon u(x_2) + \varphi (x_2)-2A (x_2-x_1)\cdot \nu  \notag \\
\geq& -\varepsilon u(x_2) - \max_{\partial \Omega} |\varphi| - 2A \text{diam}(\Omega).
\end{align}
Hence
\begin{align}\label{3.6}
\min_{\overline{\Omega}} (\varepsilon u) \geq \varepsilon \min_{\overline{\Omega}} (u - A|x-x_1|^2) \geq & \varepsilon u (x_2) - A|x_2-x_1|^2 \notag \\
\geq& - \max_{\partial \Omega} |\varphi| - 2A \text{diam}(\Omega)- A \text{diam}(\Omega)^2.
\end{align}

\end{proof}

\subsection{Global Gradient estimate}

In this subsection, we prove the global gradient estimate (independent of $\varepsilon$), using a similar argument of complex Monge-Amp\`ere equation in Li \cite{L94}.

\begin{theorem} \label{th3.2}
Suppose $\Omega \subset \mathbb{R}^n$ is a $C^3$ strictly convex domain, $\alpha_l(x) \in C^1(\overline{\Omega})$ with $l=0, 1, \cdots, k-1$ are positive functions and $\varphi \in C^2(\partial \Omega)$, and $u \in C^3(\Omega)\cap C^2(\overline \Omega)$ is the $k$-admissible solution of the equation \eqref{1.3} with $\varepsilon >0$ sufficiently small, then we have
\begin{align}\label{3.7}
\sup_\Omega |D u|  \leq M_1,
\end{align}
and
\begin{align}\label{3.8}
\sup_{\Omega} | u - \frac{1}{|\Omega|} \int_\Omega u|  \leq \overline{M_0},
\end{align}
where $M_1$ and $\overline{M_0}$ depend on $n$, $k$, $\Omega$, $|\varphi|_{C^2}$, $\inf \alpha_l$ and $|\alpha_l|_{C^1}$.
\end{theorem}

\begin{proof}
We just need to prove \eqref{3.7}, and then \eqref{3.8} holds directly from \eqref{3.7}.

In order to prove \eqref{3.7}, it suffices to prove
\begin{align}\label{3.9}
D_{\xi}u(x) \leq  M_1, \quad \forall (x, \xi) \in \overline{\Omega }\times \mathbb{S}^{n-1}.
\end{align}

For any $(x, \xi) \in \overline{\Omega }\times \mathbb{S}^{n-1}$, denote
\begin{align}\label{4.3}
W(x, \xi)=D_{\xi}u(x)-\langle \nu, \xi \rangle (- \varepsilon u + \varphi(x))+ \varepsilon^2 u^2+K|x|^2,
\end{align}
where $K$ is a large constant to be determined later, and $\nu$ is a $C^2(\overline{\Omega})$ extension of the outer unit normal vector field on $\partial \Omega$.

 Assume $W$ achieves its maximum at $(x_0, \xi_0) \in \overline{\Omega }\times \mathbb{S}^{n-1}$. It is easy to see $D_{\xi_0} u (x_0) >0$. We claim $x_0 \in \partial \Omega$. Otherwise, if $x_0 \in \Omega$, we will get a contradiction in the following.

 Firstly, we rotate the coordinates such that $D^2 u (x_0)$ is diagonal. It is easy to see $\{G^{ij} \}$ is diagonal. For fixed $\xi = \xi_0$, $W(x, \xi_0)$ achieves its maximum at the same point $x_0 \in \Omega$ and we can easily get at $x_0$
\begin{align}\label{4.4}
0\ge G^{ii} \partial_{ii} W=&G^{ii}\big[u_{ii\xi_0} -\langle\nu, \xi_0\rangle_{ii} (- \varepsilon u + \varphi)- \langle \nu, \xi_0\rangle (- \varepsilon u_{ii} + \varphi_{ii}) \notag \\
&\quad \quad- 2\langle \nu, \xi_0 \rangle_{i} (- \varepsilon u_{i} + \varphi_{i}) +2 \varepsilon^2 u_i^2+2\varepsilon^2 u u_{ii}+ K \big] \notag \\
=& D_{\xi_0} \alpha_{k-1} + \sum\limits_{l=0}^{k-2} D_{\xi_0} \alpha_l \frac{\sigma_l}{\sigma_{k-1}} + G^{ii}\big[2 \varepsilon^2 u_i^2 + 2 \langle\nu, \xi_0\rangle_{i} \varepsilon u_{i}\big] \notag \\
&+  G^{ii} u_{ii} \big[ \varepsilon \langle \nu, \xi_0 \rangle +2\varepsilon^2 u \big]  \notag \\
&+ G^{ii}\big[K  -\langle\nu, \xi_0\rangle_{ii} (- \varepsilon u + \varphi) - \langle\nu, \xi_0\rangle \varphi_{ii} - 2\langle\nu, \xi_0\rangle_{i} \varphi_{i} \big]  \notag \\
\geq& -|D \alpha_{k-1}|  - \sum\limits_{l=0}^{k-2} |D \alpha_l| C(n,k,\inf \alpha_l) - (n-k+1) |D \langle \nu, \xi_0\rangle|^2 \notag \\
&-  C(n, k, \sum \sup \alpha_l) \big[1 +2M_0 \big]  \notag \\
&+ \frac{n-k+1}{k} \big[K  - |D^2 \langle \nu, \xi_0 \rangle| (M_0 + |\varphi|) - |D^2 \varphi| - 2 |D \langle \nu, \xi_0 \rangle| |D\varphi| \big] \notag \\
>& 0,
\end{align}
where $K$ is large enough, depending only on $n$, $k$, $\Omega$, $M_0$, $|\varphi|_{C^2}$ and $\alpha_l$. This is a contradiction.

So $x_0 \in \partial \Omega$. Then we continue our proof in the following three cases.

(a) If $\xi_0$ is normal at $x_0\in \partial \Omega,$ then
\begin{align}
W(x_0,\xi_0)= \varepsilon ^2 u^2+ K |x_0|^2 \le C. \notag
\end{align}
Then we can easily obtain \eqref{3.9}.

(b) If $\xi_0$ is non-tangential at $x_0\in \partial \Omega$, then we can write $\xi_0=\alpha \tau + \beta \nu$,
where $\tau \in \mathbb{S}^{n-1}$ is tangential at $x_0$, that is $\langle\tau, \nu\rangle=0$, $\alpha =\langle \xi_0, \tau \rangle >0$, $\beta= \langle \xi_0, \nu \rangle <1$, and $\alpha^2+\beta^2=1$. Then we have
\begin{align}\label{3.12}
W(x_0, \xi_0)=&\alpha D_{\tau}u + \varepsilon ^2 u^2 + K |x_0|^2 \notag \\
\le &\alpha W(x_0, \xi_0)+(1-\alpha)(\varepsilon ^2 u^2 + K |x_0|^2),
\end{align}
so
\begin{align}
W(x_0, \xi_0)\le  \varepsilon ^2 u^2 + K |x_0|^2  \leq C. \notag
\end{align}
Then we can easily get \eqref{3.9}.

(c) If $\xi_0$ is tangential at $x_0\in \partial \Omega$, we may assume that the outer normal direction of $\Omega$ at $x_0$ is $(0,\cdots,0,1)$. By a rotation, we assume that $\xi_0=(1, \cdots, 0)=e_1$. Then we have
\begin{align}
0 \leq D_{\nu}W(x_0, \xi_0)=&D_{\nu}D_1u-D_{\nu}\langle\nu, \xi_0\rangle(- \varepsilon u + \varphi)+ 2 u \cdot D_{\nu}u+K D_{\nu} |x_0|^2  \notag \\
\leq& D_{\nu}D_{1}u + C_1 \notag\\
=& D_{1}D_{\nu}u - D_1\nu_kD_ku + C_1.
\end{align}
By the boundary condition, we know
\begin{align}
 D_{1}D_{\nu}u =  D_{1}(- \varepsilon u + \varphi) \leq  D_1 \varphi.
\end{align}
Following the argument of \cite{L94}, we can get
\begin{align}
 - D_1\nu_k D_k u  \leq  - \kappa_{min} W(x_0, \xi_0) + C_2,
\end{align}
where $\kappa_{min}$ is the minimum principal curvature of $\partial \Omega$. So
\begin{align}
 W(x_0, \xi_0) \leq \frac{C_1 + |D \varphi|+ C_2}{\kappa_{min}}.
\end{align}
Then we can conclude \eqref{3.9}.

\end{proof}

\subsection{Global Second derivatives estimates}

In this subsection, we prove the global second derivatives estimate (independent of $\varepsilon$), following the ideas of Lions-Trudinger-Urbas \cite{LTU86}, Ma-Qiu \cite{MQ15} and Chen-Zhang \cite{CZ16}.

\begin{theorem} \label{th3.3}
Suppose $\Omega \subset \mathbb{R}^n$ is a $C^4$ strictly convex domain, $\alpha_l(x) \in C^2(\overline{\Omega})$ with $l=0, 1, \cdots, k-1$ are positive functions and $\varphi \in C^3(\partial \Omega)$, and $u \in C^4(\Omega)\cap C^3(\overline \Omega)$ is the $k$-admissible solution of the equation \eqref{1.3} with $\varepsilon >0$ sufficiently small, then we have
\begin{align}\label{3.17}
\sup_\Omega |D^2 u|  \leq M_2,
\end{align}
where $M_2$ depends on $n$, $k$, $\Omega$, $|\varphi|_{C^3}$, $\inf \alpha_l$ and $|\alpha_l|_{C^2}$.
\end{theorem}

\begin{proof}
In the following, we divide the proof of Theorem \ref{th3.3} into three steps. In step one, we reduce global second derivatives to double normal second derivatives on boundary, then we prove the lower estimate of double normal second derivatives on the boundary in step two, and at last we prove the upper estimate of double normal second derivatives on the boundary.

\textbf{Step 1. Prove} $\sup_{\Omega}|D^2u|\leq C(1+\max_{\partial \Omega}|u_{\nu\nu}|)$.

Following the idea of Lions-Trudinger-Urbas \cite{LTU86}, we assume $0\in \Omega$, and consider the auxiliary function
\begin{equation}
v(x,\xi) = u_{\xi\xi} - v'(x,\xi) +K |x|^2 + |Du|^2,
\end{equation}
where $v'(x,\xi) = 2 (\xi\cdot \nu)\xi'(D \varphi - \varepsilon D u - u_m D\nu^m)=a^mu_m+b$, $\nu = (\nu^1, \cdots, \nu^{n}) \in \mathbb{S}^{n-1}$ is a $C^3(\overline{\Omega})$ extension of the outer unit normal vector field on $\partial \Omega$, $\xi'=\xi-(\xi\cdot \nu)\nu$, $a^m=2(\xi\cdot \nu)(- \varepsilon \xi'^m -\xi'^iD_i\nu^m)$,$b=2(\xi\cdot\nu)\xi'^m\varphi_{m}$, and $K>0$ is to be determined later.

For any $x \in \Omega$, we rotate the coordinates such that $D^2 u (x)$ is diagonal, and then $\{G^{ij} \}$ is diagonal. For any fixed $\xi \in \mathbb{S}^{n-1}$, we have
\begin{align}
G^{ii}v_{ii}=& G^{ii} [u_{ii\xi\xi}-D_{ii}a^mu_m- 2 D_i a^mu_{mi} -a^mu_{iim}-D_{ii}b+ 2 K +2 u_{ii}^2+ 2 u_m u_{iim}] \notag \\
=&(\alpha_{k-1})_{\xi\xi} - 2\sum_{l=0}^{k-2} (\alpha_l)_\xi {G_l}^{ii} u_{ii\xi} - \sum_{l=0}^{k-2} (\alpha_{l})_{\xi\xi} G_l-G^{ij,rs}u_{ij\xi}u_{rs\xi}  \notag \\
&+ G^{ii} [2K -D_{ii}a^mu_m -D_{ii}b ]+ G^{ii} [2 u_{ii}^2- 2 D_i a^iu_{ii}]  \notag  \\
&+(-a^m+ 2 u_m) [(\alpha_{k-1})_m - \sum_{l=0}^{k-2} (\alpha_l)_m {G_l}]  \notag \\
\geq& G^{ii} [2K -C_2]-C_1- 2\sum_{l=0}^{k-2} (\alpha_l)_\xi {G_l}^{ii} u_{ii\xi} -G^{ij,rs}u_{ij\xi}u_{rs\xi}  \notag \\
\geq& \frac{n-k+1}{k}[2K -C_2]-C_1-C_3  >0,
\end{align}
where $K$ is large enough, and we used the fact
\begin{align}
&- 2\sum_{l=0}^{k-2} (\alpha_l)_\xi {G_l}^{ii} u_{ii\xi} -G^{ij,rs}u_{ij\xi}u_{rs\xi} \notag \\
\geq& - 2\sum_{l=0}^{k-2} (\alpha_l)_\xi {G_l}^{ii} u_{ii\xi} - \sum_{l=0}^{k-2} \alpha_l {G_l}^{ij,rs}u_{ij\xi}u_{rs\xi}  \notag \\
=& - 2\sum_{l=0}^{k-2} (\alpha_l)_\xi {G_l}^{ii} u_{ii\xi} \notag  \\
&- \sum_{l=0}^{k-2} \alpha_l \Big[\frac{{k - 1 - l}}{{\left( {\frac{{\sigma _{k - 1} }}{{\sigma _l }}} \right)^{\frac{{k - l}}{{k - 1 - l}}} }}\frac{{\partial ^2 \left( {\frac{{\sigma _{k - 1} }}{{\sigma _l }}} \right)^{\frac{1}{{k - 1 - l}}} }}{{\partial u_{ij} \partial u_{rs} }} - \frac{{k - l}}{{k - 1 - l}}\frac{1}{{G_l }}G_l ^{ij} G_l ^{rs}\Big]u_{ij\xi}u_{rs\xi}  \notag \\
\geq& \sum_{l=0}^{k-2} \frac{{k - 1 - l}}{{k - l}}\frac{{(\alpha _l )_\xi  ^2 }}{{\alpha _l }}G_l \notag \\
\geq& -C_3. \notag
\end{align}
So $v(x,\xi)$ attains its maximum on $\partial \Omega$. We can assume $\max _{\bar \Omega\times \mathbb{S}^{n-1}} v(x,\xi)$ attains at $(x_0,\xi_0)\in \partial\Omega\times \mathbb{S}^{n-1}$.

Then we continue our proof in the following two cases following the idea of \cite{L94}.

Case a: $\xi_0$ is tangential to $\partial\Omega$ at $x_0$.

By the Hopf Lemma, we have
\begin{align}
0\leq v_{\nu} =& u_{\xi_0\xi_0\nu} -D_{\nu}a^m u_m  -a^mu_{m\nu}-b_{\nu}+ 2K(x\cdot \nu)+2u_m u_{m\nu}  \notag \\
\leq& u_{\xi_0\xi_0\nu} + (2u_m -a^m ) u_{m\nu} + C.
\end{align}
Following the argument in \cite{MQ15}, we can get
\begin{equation}
u_{\xi_0 \xi_0 \nu}\leq - \kappa_{\min}u_{\xi_0 \xi_0 } + C(1+|u_{\nu\nu}|),
\end{equation}
and
\begin{equation}
|u_{m\nu} |\leq  C, \quad m =1, \cdots, n.
\end{equation}
Therefore we have
\begin{equation}
u_{\xi_0 \xi_0}\leq C(1+|u_{\nu\nu}|).
\end{equation}

Case b: $\xi_0$ is non-tangential to $\partial\Omega$ at $x_0$.

We write $\xi=\hat{\alpha}\tau+\hat{\beta} \nu$, where $\hat{\alpha}=\xi\cdot\tau$, $\tau\cdot\nu=0$, $|\tau|=1$, $\hat{\beta}=\xi\cdot\nu\neq0$ and $\hat{\alpha}^2+\hat{\beta}^2=1.$ Then we have
\begin{align*}
u_{\xi\xi}=&\hat{\alpha}^2u_{\tau\tau}+\hat{\beta}^2u_{\nu\nu}+2\hat{\alpha}\hat{\beta} u_{\tau\nu}  \\
=&\hat{\alpha}^2u_{\tau\tau}+\hat{\beta}^2u_{\nu\nu}  + 2  (\xi_0 \cdot \nu)[\xi_0 - (\xi_0 \cdot \nu) \nu] [D \varphi - \varepsilon D u- D_m u D \nu^m],
\end{align*}
hence
\begin{equation}
v(x_0,\xi)=\hat{\alpha}^2 v(x_0,\tau)+\hat{\beta}^2v(x_0,\nu)
\leq \hat{\alpha}^2 v(x_0,\xi)+\hat{\beta}^2v(x_0,\nu),
\end{equation}
hence \begin{equation}
v(x_0,\xi)\leq v(x_0,\nu)\leq C(1+\max_{\partial \Omega}|u_{\nu\nu}|).
\end{equation}

\textbf{Step 2. Prove} $\min_{\partial \Omega} u_{\nu\nu} \geq -C$.

We assume $\min_{\partial\Omega}u_{\nu\nu}<0$. Also if $-\min_{\partial \Omega}u_{\nu\nu}<\max_{\partial\Omega}u_{\nu\nu}$, that is $\max_{\partial\Omega}|u_{\nu\nu}|=\max_{\partial\Omega}u_{\nu\nu}$, we shall deal this case in the next step. Thus we shall assume $-\min_{\partial\Omega}u_{\nu\nu}\geq \max_{\partial\Omega}u_{\nu\nu}$, that is $\max_{\partial\Omega}|u_{\nu\nu}|=-\min_{\partial\Omega}u_{\nu\nu}$. Denote $M:=-\min_{\partial \Omega} u_{\nu\nu}>0$ and $\bar x_0\in \partial\Omega$ such that $\min_{\partial\Omega}u_{\nu\nu}=u_{\nu\nu}(\bar x_0)$.

We consider the following test function in $\Omega_\mu = \{ x \in \Omega: 0 < d(x) < \mu\}$ ($d$ is the distance function of $\Omega$, and $\mu$ is a small universal constant)
\begin{equation}\label{011401}
P(x)= (1+\beta d)[Du \cdot (-Dd)+ \varepsilon  u -\varphi(x)]-(A+\frac{1}{2}M)d,
\end{equation}
where $\beta$ and $A$ are positive constants to be determined later.

On $\partial \Omega$, $P(x)=0$, and on $\partial \Omega_{\mu} \setminus \partial\Omega$, we have $d=\mu$ and
\begin{align*}
P(x) \leq (1+\beta \mu)[|Du| + |\varepsilon  u| +  |\varphi(x)|] - A \mu  \leq 0,
\end{align*}
since we take $A$ big enough. So on $\partial \Omega_{\mu},$ we have $P\leq 0$.  In the following, we want to prove $P$ attains its maximum only on $\partial \Omega$. Then we can get for any $x_0 \in \partial \Omega$
\begin{align}\label{5.34}
0 \leq P_\nu (x_0)  =&  [u_{\nu \nu}(x_0) - \sum\limits_m {u_m d_{m\nu} }   + \varepsilon u_\nu  - \varphi_\nu] + (A + \frac{1} {2}M)  \notag\\
\leq& u_{\nu \nu}(x_0) + |D u| |D^2 d| + \varepsilon |D u|  + |D \varphi| + A + \frac{1} {2}M,
\end{align}
which finishes the proof of Step 2.

To prove $P$ attains its maximum only on $\partial \Omega$, we assume $P$ attains its maximum at some point $\bar x_0 \in \Omega_\mu$ by contradiction. Rotating the coordinates, we can assume $D^2 u(\bar x_0)$ is diagonal, and then so is $\{ G^{ij} \}$. In the following, all the calculations are at $ \bar x_0$.

Firstly, we have
\begin{align} \label{3.28}
0= P_i(\bar x_0) = - (1+\beta d)u_{ii}d_i -(A+\frac{1}{2}M)d_i + O(1),
\end{align}
and
\begin{align} \label{3.29}
0\geq& G^{ii}P_{ii}(\bar x_0) \notag \\
=& G^{ii}\Big[ -2\beta u_{ii}{d_i}^2 + (1+\beta d)[-(u_{mii}d_m+2u_{ii}d_{ii}) + \varepsilon u_{ii}]- (A+\frac{1}{2}M)d_{ii} + O(1) \Big] \notag \\
\geq &  -2\beta G^{ii}u_{ii}{d_i}^2 - 2(1+\beta d) G^{ii}u_{ii}d_{ii} +(A+\frac{1}{2}M)c_0 - C,
\end{align}
where we used the fact
\begin{align} \label{3.30}
- G^{ii}d_{ii} \geq& \frac{{\partial \left( {\frac{{\sigma _k }}{{\sigma _{k - 1} }}} \right)}}{{\partial u_{ii} }} (- d_{ii} ) \notag \\
\geq & c(n,k) \kappa_{min} \sum_{i\ne m_0} \frac{{\partial \left( {\frac{{\sigma _k }}{{\sigma _{k - 1} }}} \right)}}{{\partial u_{ii} }} \geq c_0(n,k, \kappa_{min} )>0,
\end{align}
for some $1 \leq m_0 \leq n$.

Denote $B = \{ i:\beta {d_i }^2  < \frac{1}{n},1 \leq i \leq n\}$ and $G = \{ i:\beta {{d_i}} ^2  \geq  \frac{1}{n},1 \leq i \leq n\}$. We choose $\beta \geq \frac{1}{\mu} > 1$, so
\begin{align}
{d_i }^2 < \frac{1}{n} = \frac{1}{n} |D d|^2, \quad i \in B.
\end{align}
It holds $ \sum_{i \in B} {d_i} ^2 < 1 = |D d|^2$, and $G$ is not empty. Hence for any $i \in G$, it holds
\begin{align}
{d_{i}} ^2 \geq \frac{1}{n\beta}. \notag
\end{align}
From \eqref{3.28}, we have
\begin{align}
u_{i i }  =  - \frac{{1}}{1+ \beta d}(A + \frac{1}{2}M) + \frac{{O(1)}}{(1+ \beta d) d_{i }}.
\end{align}
So when $A$ is large enough, we can get
\begin{align}
 u_{i i }  \leq  - \frac{A+M}{5}, \quad \forall \quad i \in G.
\end{align}
Also there is an $i_0  \in G$ such that
\begin{align}
{d_{i_0} }^2 \geq \frac{1}{n} |D d|^2 = \frac{1}{n}.
\end{align}
From \eqref{3.29}, we have
\begin{align}
0\geq&G^{ii}P_{ii}(\bar x_0)  \notag \\
\geq& -2\beta\sum_{i\in G} G^{ii}u_{ii}{d_i}^2  - 2\beta\sum_{i\in B} G^{ii}u_{ii}{d_i}^2- 2(1+\beta d)\sum_{u_{ii}<0} G^{ii}u_{ii}d_{ii} \notag \\
& +(A+\frac{1}{2}M)c_0-C  \notag \\
\geq & -\frac{2\beta}{n}G^{i_0i_0}u_{i_0i_0} -C +(\frac{2}{n}+4\kappa_{\max})\sum_{u_{ii}<0}G^{ii}u_{ii},
\end{align}
where we used the facts
\begin{equation}
-2\beta\sum_{i\in G} G^{ii}u_{ii}{d_i}^2 \geq -2\beta G^{i_0i_0}u_{i_0i_0}{d_{i_0}}^2 \geq -\frac{2\beta}{n}G^{i_0i_0}u_{i_0i_0},
\end{equation}
and
\begin{align*}
- 2\beta\sum_{i\in B} G^{ii}u_{ii}{d_i}^2 &\geq -2\beta\sum_{i\in B, u_{ii}>0}G^{ii}u_{ii}{d_i}^2 \geq -\frac{2}{n}\sum_{i\in B, u_{ii}>0}G^{ii}u_{ii}\\
&\geq -\frac{2}{n}\sum_{u_{ii}>0}G^{ii}u_{ii} =-\frac{2}{n}[\sum G^{ii}u_{ii}-\sum_{u_{ii}<0}G^{ii}u_{ii}]\\
&\geq \frac{2}{n}\sum_{u_{ii}<0}G^{ii}u_{ii}.
\end{align*}

Therefore, we have
\begin{align}
0\geq&G^{ii}P_{ii}(\bar x_0) \notag \\
\geq & \frac{2\beta}{n}(\frac{A+M}{5})c(n,k)\sum_{i=1}^n G^{ii} -C - (\frac{2}{n}+4\kappa_{\max})\sum_{u_{ii}<0} G^{ii} |D^2 u| \notag \\
\geq & \frac{2\beta}{n}(\frac{A+M}{5})c(n,k) \frac{n-k+1}{k} -C - (\frac{2}{n}+4\kappa_{\max})(n-k+1)C(1+M) \notag \\
> & 0,
\end{align}
by taking $\beta$ big enough. This is a contradiction. So $P$ attains its maximum only on $\partial \Omega$.

\textbf{Step 3. Prove} $\max_{\partial \Omega} u_{\nu\nu} \leq C$.

Similar with Step 2, we can assume $\max_{\partial\Omega}u_{\nu\nu}>0$, and $-\min_{\partial\Omega}u_{\nu\nu}\leq \max_{\partial\Omega}u_{\nu\nu}$, that is $\max_{\partial\Omega}|u_{\nu\nu}|=\max_{\partial\Omega}u_{\nu\nu}$. Denote $M :=\max_{\partial\Omega} u_{\nu\nu}>0$ and $\tilde{x_0}\in \partial\Omega$ such that $\max_{\partial\Omega}u_{\nu\nu}=u_{\nu\nu}(\tilde{x_0})$.

We consider the following test function in $\Omega_\mu = \{ x \in \Omega: 0 < d(x) < \mu\}$ ($d$ is the distance function of $\Omega$, and $\mu$ is a small universal constant)
\begin{equation}\label{3.38}
\widetilde{P}(x)= (1+\beta d)[Du \cdot (-Dd)+ \varepsilon  u -\varphi(x)] + (A+\frac{1}{2}M)d,
\end{equation}
where $\beta$ and $A$ are positive constants to be determined later.

On $\partial \Omega$, $\widetilde{P}(x)=0$, and on $\partial \Omega_{\mu} \setminus \partial\Omega$, we have $d=\mu$ and
\begin{align*}
\widetilde{P}(x) \geq -(1+\beta \mu)[|Du| + |\varepsilon  u| +  |\varphi(x)|] + A \mu  \geq 0,
\end{align*}
since we take $A$ big enough. So on $\partial \Omega_{\mu},$ we have $\widetilde{P} \geq 0$.  In the following, we want to prove $\widetilde{P}$ attains its minimum only on $\partial \Omega$. Then we can get for any $x_0 \in \partial \Omega$
\begin{align}
0 \geq \widetilde{ P}_\nu (x_0)  =&  [u_{\nu \nu}(x_0) - \sum\limits_m {u_m d_{m\nu} }   + \varepsilon u_\nu  - \varphi_\nu] - (A + \frac{1} {2}M)  \notag\\
\geq& u_{\nu \nu}(x_0) - |D u| |D^2 d| - \varepsilon |D u|  - |D \varphi| -( A + \frac{1} {2}M),
\end{align}
which finishes the proof of Step 3.

To prove $\widetilde{P}$ attains its minimum only on $\partial \Omega$, we assume $\widetilde{P}$ attains its minimum at some point $\tilde{x_0} \in \Omega_\mu$ by contradiction. Rotating the coordinates, we can assume $D^2 u(\tilde{x_0})$ is diagonal, and then so is $\{ G^{ij} \}$. In the following, all the calculations are at $ \tilde{x_0}$.

Firstly, we have
\begin{align} \label{3.40}
0= \widetilde{P}_i(\tilde{x_0}) = - (1+\beta d)u_{ii}d_i +(A+\frac{1}{2}M)d_i + O(1),
\end{align}
and
\begin{align} \label{3.41}
0\leq& G^{ii}\widetilde{P}_{ii}(\tilde{x_0}) \notag \\
=& G^{ii}\Big[ -2\beta u_{ii}{d_i}^2 + (1+\beta d)[-(u_{mii}d_m+2u_{ii}d_{ii}) + \varepsilon u_{ii}]+ (A+\frac{1}{2}M)d_{ii} + O(1) \Big] \notag \\
\leq &  -2\beta G^{ii}u_{ii}{d_i}^2 - 2(1+\beta d) G^{ii}u_{ii}d_{ii} -(A+\frac{1}{2}M)c_0 + C.
\end{align}

Denote $B = \{ i:\beta {d_i} ^2  < \frac{1}{n},1 \leq i \leq n\}$ and $G = \{ i:\beta {d_i} ^2  \geq  \frac{1}{n},1 \leq i \leq n\}$. We choose $\beta \geq \frac{1}{\mu} > 1$, so
\begin{align}\label{3.42}
{d_i} ^2 < \frac{1}{n} = \frac{1}{n} |D d|^2, \quad i \in B.
\end{align}
It holds $ \sum_{i \in B} {d_i} ^2 < 1 = |D d|^2$, and $G$ is not empty. Hence for any $i \in G$, it holds
\begin{align}\label{3.43}
{d_{i}} ^2 \geq \frac{1}{n\beta}.
\end{align}
and from \eqref{3.40}, we have
\begin{align}
u_{i i }  =   \frac{{1}}{1+ \beta d}(A + \frac{1}{2}M) + \frac{{O(1)}}{(1+ \beta d) d_{i }}.
\end{align}
So when $A$ is large enough, we can get
\begin{align}
 \frac{3A}{5}+ \frac{2M}{5} \leq  u_{i i }  \leq \frac{6A}{5}+ \frac{M}{2}, \quad \forall \quad i \in G.
\end{align}
Also there is an $i_0  \in G$ such that
\begin{align}
{d_{i_0}} ^2 \geq \frac{1}{n} |D d|^2 = \frac{1}{n}.
\end{align}
From \eqref{3.41}, we have
\begin{align} \label{3.47}
0\leq& G^{ii}\widetilde{P}_{ii}( \tilde{x_0})  \notag \\
\leq& -2\beta\sum_{i\in G} G^{ii}u_{ii}{d_i}^2  - 2\beta\sum_{i\in B} G^{ii}u_{ii}{d_i}^2- 2(1+\beta d)\sum_{u_{ii}>0} G^{ii}u_{ii}d_{ii} \notag \\
& -(A+\frac{1}{2}M)c_0 + C  \notag \\
\leq & -\frac{2\beta}{n}G^{i_0i_0}u_{i_0i_0} +(\frac{2}{n}+4\kappa_{\max})\sum_{u_{ii}> 0}G^{ii}u_{ii} -(A+\frac{1}{2}M)c_0 + C,
\end{align}
where we used the facts
\begin{equation}
-2\beta\sum_{i\in G} G^{ii}u_{ii}{d_i}^2 \leq -2\beta G^{i_0i_0}u_{i_0i_0}{d_{i_0}}^2 \leq -\frac{2\beta}{n}G^{i_0i_0}u_{i_0i_0},
\end{equation}
and
\begin{align*}
- 2\beta\sum_{i\in B} G^{ii}u_{ii}d_i^2 &\leq -2\beta\sum_{i\in B, u_{ii}<0}G^{ii}u_{ii}d_i^2 \leq -\frac{2}{n}\sum_{i\in B, u_{ii}<0}G^{ii}u_{ii}\\
&\leq -\frac{2}{n}\sum_{u_{ii}<0}G^{ii}u_{ii} =-\frac{2}{n}[\sum G^{ii}u_{ii}-\sum_{u_{ii}>0}G^{ii}u_{ii}]\\
&\leq  \frac{2}{n}\sum_{u_{ii}>0}G^{ii}u_{ii}.
\end{align*}

In the following , we divide into three cases to prove the result. Without loss of generality, we can assume that $i_0=1\in G$, and $u_{22}\geq \cdots\geq u_{nn}$.

Case I: $u_{nn}>0$.

In this case, we have
\begin{align} \label{3.49}
0\leq& G^{ii}\widetilde{P}_{ii}( \tilde{x_0})  \notag \\
\leq & -\frac{2\beta}{n}G^{i_0i_0}u_{i_0i_0} +(\frac{2}{n}+4\kappa_{\max})\sum G^{ii}u_{ii} -(A+\frac{1}{2}M)c_0 + C \notag \\
\leq & (\frac{2}{n}+4\kappa_{\max}) C(n,k, \sum \sup \alpha_l) -(A+\frac{1}{2}M)c_0 + C \notag \\
<& 0,
\end{align}
by taking $A$ large enough. This is a contradiction.

Case II: $u_{nn}<0$ and $-u_{nn}<\frac{c_0}{10n(4\kappa_{\max}+\frac{2}{n})}u_{11}.$
\begin{align} \label{3.50}
(4\kappa_{\max}+\frac{2}{n})\sum_{u_{ii}>0}G^{ii}u_{ii} =&(4\kappa_{\max}+\frac{2}{n})[\sum_{i=1}^n G^{ii}u_{ii}-\sum_{u_{ii}<0}G^{ii}u_{ii}] \notag \\
\leq &(4\kappa_{\max}+\frac{2}{n})[\sum_{i=1}^n G^{ii}u_{ii}-u_{nn}\sum_{i=1}^nG^{ii}]  \notag \\
\leq& C+\frac{c_0}{10n}u_{11}\sum_{i=1}^n G^{ii}   \notag \\
\leq & C+ \frac{c_0}{10}(\frac{6A}{5}+\frac{M}{2}).
\end{align}

Hence combining \eqref{3.47} and \eqref{3.50}, we have
\begin{align} \label{3.49}
0\leq& G^{ii}\widetilde{P}_{ii}( \tilde{x_0})  \notag \\
\leq & (\frac{2}{n}+4\kappa_{\max})\sum_{u_{ii}>0} G^{ii}u_{ii} -(A+\frac{1}{2}M)c_0 + C \notag \\
\leq &  \frac{c_0}{10}(\frac{6A}{5}+\frac{M}{2}) -(A+\frac{1}{2}M)c_0 + C \notag \\
<& 0,
\end{align}
by taking $A$ large enough. This is a contradiction.

Case III: $u_{nn}<0$ and $-u_{nn}\geq\frac{c_0}{10n(4\kappa_{\max}+\frac{2}{n})}u_{11}.$

We have $u_{11}\geq \frac{3A}{5}+\frac{2M}{5}$ and $u_{22}\leq C(1+M)$. So $u_{11}\geq \frac{2}{5C}u_{22}$. Let $\delta=\frac{2}{5C}$ and $\epsilon=\frac{c_0}{10n(4\kappa_{\max}+\frac{2}{n})}$, by Lemma \ref{lem2.9}, we have
\begin{equation}\label{3.52}
G^{11} \geq c_2\sum_{i=1}^n G^{ii}.
\end{equation}
Hence from \eqref{3.47} and \eqref{3.52}, we have
\begin{align} \label{3.49}
0\leq& G^{ii}\widetilde{P}_{ii}( \tilde{x_0})  \notag \\
\leq & -\frac{2\beta}{n}G^{11}u_{11} +(\frac{2}{n}+4\kappa_{\max})\sum_{u_{ii}>0} G^{ii}u_{ii} -(A+\frac{1}{2}M)c_0 + C \notag \\
\leq &  -\frac{2\beta}{n}c_2(\frac{3A}{5}+\frac{2M}{5})\sum_{i=1}^nG^{ii}+ (4\kappa_{\max}+\frac{2}{n})C(1+M)\sum_{i=1}^n G^{ii} \notag \\
<& 0,
\end{align}
by taking $\beta$ and $A$ big enough. This is a contradiction.

The proof is finished.
\end{proof}

\section{Existence}

 In this section, we prove Theorem \ref{th1.1}.

Firstly, we prove the existence of the $k$-admissible solution of the approximation equation \eqref{1.3} for any small $\varepsilon >0$.

For the Neumann problem of approximation equation \eqref{1.3}, we have established the $C^2$ estimates in Section 3. By the global $C^2$ a priori estimates, we obtain that the equation \eqref{1.3} are uniformly elliptic in $\overline \Omega$.
Due to the concavity of the operator $G$, we can get the global H\"{o}lder estimates of second derivative following the discussions in \cite{LT86}, that is, we can get
\begin{align}\label{4.1}
|u|_{C^{2,\alpha}(\overline \Omega)} \leq C,
\end{align}
where $C$ and $\alpha$ depend on $n$, $k$, $\Omega$, $\varepsilon$, $\inf \alpha_l$, $|\alpha_l|_{C^2}$ and $|\varphi|_{C^3}$.

Applying the method of continuity (see \cite{GT}, Theorem 17.28), we can show the existence of the classical solution, and the solution is unique by Hopf lemma. By the standard regularity theory of uniformly elliptic partial differential equations, we can obtain the high order regularity.

Now, we start to prove Theorem \ref{th1.1}.

By the above argument, we know there exists a unique $k$-admissible solution $u^\varepsilon \in C^{3, \alpha}(\overline \Omega)$ to \eqref{1.3} for any small $\varepsilon >0$. Let $v^\varepsilon =u^\varepsilon - \frac{1}{|\Omega|} \int_\Omega u^\varepsilon$, and it is easy to know $v^\varepsilon$ satisfies
\begin{align}
\left\{ \begin{array}{l}
\sigma _k (D^2 v^\varepsilon) = \sum\limits_{l=0}^{k-1} \alpha_l \sigma _l (D^2 v^\varepsilon),  \quad \text{in} \quad \Omega,\\
(v^\varepsilon)_\nu = - \varepsilon v^\varepsilon - \frac{1}{|\Omega|} \int_\Omega \varepsilon u^\varepsilon + \varphi(x),\quad \text{on} \quad \partial \Omega.
 \end{array} \right.
\end{align}
By the global gradient estimate \eqref{3.7}, it is easy to know $\varepsilon \sup |D u^\varepsilon | \rightarrow 0$. Hence there is a constant $c$ and a function $v \in C^{2}(\overline \Omega)$, such that $-\varepsilon u^\varepsilon \rightarrow c$, $-\varepsilon v^\varepsilon \rightarrow 0$, $-\frac{1}{|\Omega|} \int_\Omega \varepsilon u^\varepsilon \rightarrow c$ and $v^\varepsilon \rightarrow v$ uniformly in $C^{2}(\overline \Omega)$ as $\varepsilon \rightarrow 0$. It is easy to verify that $v$ is a $k$-admissible solution of
\begin{align}
\left\{ \begin{array}{l}
\sigma _k (D^2 v) = \sum\limits_{l=0}^{k-1} \alpha_l \sigma _l (D^2 v),  \quad \text{in} \quad \Omega,\\
v_\nu = c + \varphi(x),\qquad \text{on} \quad \partial \Omega.
 \end{array} \right.
\end{align}
If there is another function $v_1 \in C^{2}(\overline \Omega)$ and another constant $c_1$ such that
\begin{align}
\left\{ \begin{array}{l}
\sigma _k (D^2 v_1) = \sum\limits_{l=0}^{k-1} \alpha_l \sigma _l (D^2 v_1),  \quad \text{in} \quad \Omega,\\
(v_1)_\nu = c_1 + \varphi(x),\qquad \text{on} \quad \partial \Omega.
 \end{array} \right.
\end{align}
Applying the maximum principle and Hopf Lemma, we can know $c = c_1$ and $v - v_1$ is a constant.
By the standard regularity theory of uniformly elliptic partial differential equations, we can obtain the high oder regularity.

\end{document}